\documentclass{amsart}
   
	\usepackage{dsfont}
	\usepackage{amssymb,amsmath,amsthm}
	\usepackage{mathrsfs}

	\usepackage[T1]{fontenc}
	\usepackage[utf8]{inputenc}
	\usepackage{tikz}
	\usetikzlibrary{shapes.geometric}

	\usepackage[shortlabels]{enumitem}
	\usepackage{mathtools}
	\usepackage{hyperref}
	\usepackage{todonotes}
	\usepackage{xfrac}
	\usepackage{booktabs}
	\usepackage{bm}

	\newtheorem{thm}{Theorem}
	\newtheorem{lemma}{Lemma}
	\newtheorem{cor}[thm]{Corollary}
	\newtheorem{numbered-cor}{Corollary}

	\theoremstyle{definition}

	\theoremstyle{remark}
	\newtheorem*{rem}{Remark}

	\renewcommand{\Pr}{\mathbb{P}}
	\renewcommand{\S}{{\mathcal{S}}}
	\newcommand{\E}{{\mathbb{E}}}
	\newcommand{\Var}{\mathbb{V}\!\textrm{ar}}
	\newcommand{\R}{{\mathbb{R}}}
	\newcommand{\G}{{\tilde{G}}}			 
	\newcommand{\C}{{\tilde{C}}}
	\newcommand{\V}{{\tilde{V}}}
	\newcommand{\Ed}{{\tilde{E}}}
	
	\newcommand{\ep}{{\varepsilon}}
	\newcommand{\domega}{\mathop{}\!\mathrm{d}\omega}
  \newcommand{\iV}{\overline{V}}
	\newcommand{\iE}{\overline{E}}

\begin{document}

\title[Max expected number of components in an online search of a graph]{Maximizing the expected number of components in an online search of a graph} 
\author[F.\,S.~Benevides, M.~Sulkowska]{Fabr{\'i}cio Siqueira Benevides \and Ma{\l}gorzata Sulkowska} 
\address{Departamento de Matem{\'a}tica, Campus do Pici, Universidade Federal do Cear{\'a}, Brazil.}
\email{fabricio@mat.ufc.br} 
\address{Wroc{\l}aw University of Science and Technology, Department of Fundamentals of Computer Science, Poland;
Universit{\'e} C{\^o}te d’Azur, CNRS, Inria, I3S, France.
}
\email{malgorzata.sulkowska@pwr.edu.pl}

\subjclass{Primary: 60G40, Secondary: 60K35} 

\keywords {optimal stopping, secretary problem, $2$-dimensional lattice, simple graph}

\date{\today} 
 
\begin{abstract} 
The following optimal stopping problem is considered. The vertices of a graph $G$ are revealed one by one, in a random order, to a selector. He aims to stop this process at a time $t$ that maximizes the expected number of connected components in the graph $\G_t$, induced by the currently revealed vertices. The selector knows $G$ in advance, but different versions of the game are considered depending on the information that he gets about $\G_t$. We show that when $G$ has $N$ vertices and maximum degree of order $o(\sqrt{N})$, then the number of components of $\G_t$ is concentrated around its mean, which implies that playing the optimal strategy the selector does not benefit much by receiving more information about $\G_t$. Results of similar nature were previously obtained by M.\,Laso{\'n} for the case where $G$ is a $k$-tree (for constant $k$). We also consider the particular cases where $G$ is a square, triangular or hexagonal lattice, showing that an optimal selector gains $cN$ components and we compute $c$ with an error less than $0.005$ in each case.
\end{abstract} 
 
\maketitle

\section{Introduction}

Let $G = (V, E)$ be a graph on $N$ vertices. Let $\S$ be the set of all permutations of~$V$. We consider the following online stopping problem. Select uniformly at random a permutation $\sigma \in \S$, say $\sigma = (\sigma_1, \sigma_2, \ldots, \sigma_N)$. The vertices of $G$ emerge, one by one, following the order given by $\sigma$. For $t \in \{1, 2, \ldots, N\}$, let $\G_t(\sigma)$, or simply $\G_t$, be the graph induced by $\{\sigma_1, \sigma_2, \ldots, \sigma_t\}$. We think about $t$ as time and at each time step a player, who knows what the graph $G$ is, must take a decision based on some information that he receives about $\G_t$: either he continues the process and reveals the next vertex or he stops the game and gains as payoff the number of connected components of $\G_t$. In case he decides to reveal another vertex, he is not allowed to go back to the previous step of the game. If the player never takes the decision to stop, the game finishes with $\G_N$ selected and the payoff is equal to the number of components of $G$.

The presented question may be treated as one of many generalizations of the celebrated \emph{secretary problem} that attracted the attention of mathematical society in the early 1960's (consult \cite{secretaryDL} and \cite{Ferguson}). In the secretary problem the player observes elements of a linear order emerging one by one in some random permutation. At a given time step he can see the order induced by the elements that have already appeared. His task is to stop the search maximizing the probability that the element that has just appeared is the maximal one in the whole order. Series of papers in which the linear order has been replaced by a partial order followed the work of Stadje~\cite{Stadje}. Optimal strategies for a particular posets as well as universal algorithms for the whole families of posets have been presented in \cite{bin_tree}, \cite{Garrod}, \cite{univ_poset}, \cite{Kozik} and \cite{Freij_Wastlund}. Kubicki and Morayne were the first ones to investigate the optimal stopping problem on a directed graph choosing a directed path as the underlying structure, \cite{dir_path}. The link between the directed path case and the classical secretary problem was given by the authors in \cite{kPaths}. Universal algorithms for graphs were formulated by Goddard et al.~in \cite{Kubiccy_Goddard} and by Sulkowska in \cite{opt_Sulkowska}. This paper continues the study of optimal stopping algorithms for graphs. However, the approach to the subject is slightly different since now the aim is to maximize the expected number of components at the stopping moment instead of the probability that the last vertex belongs to some previously defined set. In turn, the study of components is another classical topic in the area of random graphs. The first paper that puts optimal stopping for graphs in the setup of counting components is \cite{kTrees} by Laso{\'n}.

One can introduce various versions of the presented stopping game, depending on exactly what information the player receives about $\G_t$. In every version, the player knows $G$ in advance and his task is to find the strategy that maximizes the expected payoff. The following three versions may serve as examples. 

\begin{enumerate}[1.]
\item Blind game. At time $t$ the selector knows only the number of vertices that have already appeared (i.e., $t$). He has no other information about the revealed structure. In fact, he gains no information during the game.
\item Partial information game. The selector can see an unlabeled graph isomorphic to $\G_t$. In particular, he knows how many edges or components are there at time $t$, but he does not know exactly which vertices of $G$ have been selected. This is a classical setup for many optimal stopping problems considered in the past (comparable to the setup in the secretary problem).
\item Full information game. The selector knows $\{\sigma_1, \sigma_2, \ldots, \sigma_t\}$, and since he knows $G$, he knows $\G_t$. Thus he gets all information that is available at time $t$.
\end{enumerate}

In~\cite{kTrees}, Laso{\'n} considers the above three versions for the case when $G$ is a $k$-tree, for some constant $k$. Surprisingly, the maximum expected payoff for a selector that plays optimally with full information is very close to the one for an optimal selector playing the blind game. In this article, we prove that a similar statement holds for any graph $G$ with $N$ vertices and maximum degree bounded from above by $o(\sqrt{N})$. Throughout the rest of the paper we are going to refer only to blind and full information games, as the expected payoff for the partial information game is between those of the other two. 

We also study the cases where $G$ is a square, a triangular or a hexagonal lattice and we provide tight estimates for the expected payoff in those three cases. The study of lattices is motivated by the relation (explained below) between our problem and the well studied site percolation problem on $2$-dimensional lattices.
Another motivation is the fact that the results from~\cite{kTrees} are stated for $k$-trees, which are maximal (with respect to the inclusion of edges) $k$-degenerate graphs and at the same time maximal graphs with treewidth equal to $k$. In contrast, $2$-dimensional lattices are also $k$-degenerate (a square lattice and a hexagonal lattice are $2$-degenerate, while a triangular lattice is $3$-degenerate) but all have unbounded treewidth. It turns out that the maximum expected payoff for $2$-dimensional lattices is smaller than the one for $k$-trees in a non-negligible way.

\section{Formal model and notation}
\label{sec:formalmodel}
Our graph $G = (V, E)$ with $V = \{v_1, v_2, \ldots, v_N\}$ is always finite, simple and undirected. The degree of $v \in V$ will be denoted by $\deg(v)$. An induced subgraph $G' = (W, E \cap W^2)$, where $W \subseteq V$, is called \emph{a component} of $G$ if it is a maximal connected induced subgraph of $G$. 
We define a probability space $(\S, {\mathcal{P}}, \Pr)$, where ${\mathcal{P}}$ is the set of all subsets of $\S$ and the probability measure is defined by $\Pr[\{\sigma\}] = 1/N!$ for any permutation $\sigma \in \S$. A \emph{stopping time} is a function $\tau: \S \to \{1, 2, \ldots, N\}$ such that its value on a permutation $\sigma$, say $t = \tau(\sigma)$, depends only on the information the selector gathered up to time $t$, which is information the selector gets about $\G_t$ (the graph induced by $\{\sigma_1, \sigma_2, \ldots, \sigma_t\}$).  E.g., in the blind game the selector gains no information during the game thus $\tau$ needs to be the same for all $\sigma \in \S$, it is simply a constant function depending only on $G$. Formally, let $\mathcal{P}_1 \subseteq \mathcal{P}_2 \subseteq \ldots \subseteq \mathcal{P}_N \subseteq \mathcal{P}$ be a sequence of $\sigma$-algebras (such a sequence is called a \emph{filtration}). A random variable $\tau: \S \to \{1, 2, \ldots, N\}$ is a stopping time with respect to a filtration $(\mathcal{P}_t)_{t=1}^{N}$ if $\tau^{-1}(t) \in \mathcal{P}_t$ for all $t \leq N$. In the blind game we define $\mathcal{P}_t = \{\emptyset, \mathcal{S}\}$ for every $t$. In the full information game, for each $t \in \{1, 2, \ldots, N\}$, the sets $A$ in $\mathcal{P}_t$ are those with the following property: if $\sigma = (\sigma_1, \ldots, \sigma_N) \in A$, then for every $\pi = (\pi_1, \ldots, \pi_N) \in \mathcal{S}$ such that $\{\sigma_1, \ldots, \sigma_t\} = \{\pi_1, \ldots, \pi_t\}$ we have $\pi \in A$.

 A stopping algorithm is any algorithm that produces a stopping time.
Let ${\mathcal{T}}$ denote the set of all stopping times (note that the definition of ${\mathcal{T}}$ depends on whether we are playing the blind or the full information version, but we omit this in the notation). Let $\C_t$ be the number of components of $\G_t$. We say that $\tau^*$ is optimal if
\[
	\tau^* = {\arg\!\max}_{\tau \in {\mathcal{T}}} {\E[\C_{\tau}]},	
\]
that is, it maximizes the expected number of components at the time it stops. Here, $\E[\C_{\tau}]$ is understood as $\frac{1}{N!}\sum_{\sigma \in \S}{\C_{\tau(\sigma)}(\sigma)}$.

This is the classical optimal stopping setup, but in our proofs it is going to be more convenient to work with a different (probability) model. Assume the graph $G$ is given and let $p \in [0, 1]$. Each vertex of $G$ is declared \emph{open} with probability~$p$ and \emph{closed} with probability $1-p$, independently of the other vertices. By $G_p$ we denote the graph induced by the set of open vertices. Let $C_p$ be the number of connected components of $G_p$. When $G$ is an infinite lattice, the problem of deciding for what values of $p$ there exists (with high probability) an infinite connected component in $G_p$ is known as \emph{site percolation}. Due to its huge number of applications this problem was overly studied by mathematicians as well as physicists (see, for example, the book~\cite{bollobas2006percolation}). Both, theoretical arguments and computer simulations were used in order to investigate percolation phenomenon, especially in the context of phase transitions. However, we have not found articles providing good estimates for $C_p$ for general values of $p$ (especially when $p$ is far from the percolation threshold). For some general estimates consult \cite{physMatch,clustersGrimmet}.

Intuitively, for $N$ sufficiently large and $t\in \{1, 2, \ldots, N\}$, letting $p = t/N$, one should expect that $\C_t$ and $C_p$ behave similarly.
For the sake of completeness, we will prove that this is indeed the case for graphs with maximum degree bounded by $o(\sqrt{N})$.  
To show that $C_p$ is concentrated around its mean we use classical concentration result from \cite{ifmUTail} known as Azuma's inequality or McDiarmid's inequality. In order to compare $C_p$ and $\C_t$, or in general $G_p$ and $\G_t$, we consider a coupling of those random variables on the probability space $(\Omega, {\mathcal{F}}, \Pr)$, where $\Omega = [0,1]^N$, $\mathcal{F}$ is the family of all Borel sets of $\Omega$, and $\Pr$ is the uniform distribution (i.e., with probability density function equal to $1$ everywhere). Given $\omega = (\omega_1, \ldots, \omega_N) \in \Omega$, we interpret $\omega_i$ as the arrival time of the vertex $v_i$ of $G$. Note that each $\omega$ induces (almost surely) a permutation $\sigma$ such that $\sigma = (v_{i_1}, v_{i_2}, \ldots, v_{i_N})$ for $\omega_{i_1} < \omega_{i_2} < \ldots < \omega_{i_N}$. If this is the case, we write $\omega \leadsto \sigma$.  Note that, in this way, the distribution that $\omega$ induces on the set of permutations $\S$ is the uniform distribution. 

For $\omega \in \Omega$, by $\G_t(\omega)$ we understand the graph induced by $\{\sigma_1, \sigma_2, \ldots, \sigma_t\}$, where $\sigma$ is induced by $\omega$. Note that, for any fixed $t \in \{1, \ldots, N\}$, we have that $\G_t(\omega)$ has the same distribution as  $\G_t(\sigma)$ where $\sigma$ is drawn directly from $(\S, {\mathcal{P}}, \Pr)$. Similarly, given $p \in [0,1]$, we can define $G_p(\omega)$ as the graph induced by the vertices $v_i$ for which $\omega_i \le p$ and such graph has the same distribution as $G_p$ defined in the previous paragraph.
Note that, with this notation,
\[
	\E\bigr[\C_{\tau}\bigl] = \int_{\Omega}{\C_{\tau(\omega)}(\omega) \domega},
\]
where $\tau(\omega)$ is naturally understood as $\tau(\sigma)$, where $\sigma$ is the permutation induced by $\omega$.

By $\V_t(\omega)$, $\Ed_t(\omega)$ and $\C_t(\omega)$ we denote, respectively, the number of vertices, edges and components of the random graph $\G_t(\omega)$. Whenever the context is clear we write $\V_t$, $\Ed_t$ and $\C_t$ for short. Similarly, we denote by $V_p$, $E_p$ and $C_p$ the analogous random variables with respect to $G_p$. Note that $\V_t$ is a constant equal to $t$, while $V_p$ follows the binomial distribution with parameters $N$ and $p$. 

For $S \subseteq \S$ let $\Omega_S = \bigcup_{\sigma \in S}\{\omega : \omega \leadsto \sigma\}$. In particular when $S = \{\sigma\}$, we simply use $\Omega_\sigma$. Note that $\Ed_t(\omega)$ is a random variable that is constant on $\Omega_\sigma$ (and the same holds for $\C_t(\omega)$). This value will be interchangeably denoted by $\Ed_t(\sigma)$ (and similarly $\C_t(\sigma)$ for $\C_t(\omega)$).

\section{Blind versus Full Information}

This whole section is devoted to proving quite a surprising result. We show that in many situations 
the maximum expected payoff while playing in the full information mode is very close to the maximum expected payoff while playing in the blind mode. Precisely speaking, we are going to prove the following theorem.

\begin{thm}
\label{thm_blind_full}
Let $G$ be a graph on $N$ vertices. Let $\tau^f$ be the optimal algorithm while playing in a full information mode and let $\tau^b$ be the optimal algorithm while playing in a blind mode. For every $\ep\in(0,1)$ there exists $N_{\ep}$ such that if $N \geq N_{\ep}$ and the maximum degree of $G$ is bounded by $D_{\ep,N} = \dfrac{\ep^2}{32} \sqrt{N}$, then
\[
\E[\C_{\tau^b}] \leq \E[\C_{\tau^f}] \leq \E[\C_{\tau^b}] + \ep N.
\] 
\end{thm}
\begin{rem}
	Note that the stated bounds are useful when $\E[\C_{\tau^b}]$ is of order $N$. Determining precisely for which families of graphs this condition holds is left as a future work. In the next section we give natural examples of graphs that fulfill this requirement.
\end{rem}

Before we prove the main theorem, we state several technical lemmas that will be helpful later on. The first one is a concentration result (known as McDiarmid's inequality) which is a version of Azuma's inequality tailored for combinatorial applications (see \cite{ineq_Azuma} and \cite{ifmUTail}).

\begin{lemma}
\label{lemma_azuma_ineq}
Let $Z_1, Z_2, \ldots, Z_M$ be independent random variables, with $Z_j$ taking values in a set $\Lambda_j$. Assume that a function $g: \Lambda_1 \times \Lambda_2 \times \ldots \times \Lambda_M \to \R$ satisfies, for some constants $b_j$, where $j \in \{1, 2, \ldots, M\}$, the following Lipschitz condition:
\begin{itemize}[label = {(L)}]
\item if two vectors $\mathbf{z}, \mathbf{z'} \in \Lambda_1 \times \Lambda_2 \times \ldots \times \Lambda_M$ differ only in $j\textsuperscript{th}$ coordinate, then $|g(\mathbf{z})-g(\mathbf{z'})| \leq b_j$.
\end{itemize}
Then the random variable $X = g(Z_1, Z_2, \ldots, Z_M)$ satisfies, for any $t \geq 0$,
\begin{align*}
\Pr[X \geq \E[X] + t] & \leq \exp\left\{\frac{-2 t^2}{\sum_{j=1}^{M} b_j^2}\right\},\\
\Pr[X \leq \E[X] - t] & \leq \exp\left\{\frac{-2 t^2}{\sum_{j=1}^{M} b_j^2}\right\}.
\end{align*}
\end{lemma}

The above lemma generalizes many known concentration bounds for sums of independent random variables. For example, compare it to Hoeffding's inequality~\cite{ineq_Azuma}. 

We will also need a concentration inequality for a binomial random variable, $X \sim \mathop{\textrm{Bin}}(N,p)$, for an additive error of order $\sqrt{N}/\ep$. In this range, it will be enough (and more convenient for us) to simply use Chebyshev's inequality.



\begin{lemma}
\label{lemma_conc_vert}
Let $X$ be a random variable following the binomial distribution with parameters $N$ and $p \in [0,1]$. Then, for every $\ep>0$,
\[
\Pr\left[|X - \E[X]| \geq 4 \sqrt{N}/\ep \right] \leq \frac{\ep^2}{64}.
\]
\end{lemma}
\begin{proof}
By Chebyshev's inequality we get
\[
\Pr\left[|X - \E[X]| \geq 4 \sqrt{N}/\ep \right] \leq \frac{\ep^2 \Var[X]}{16 N} = \frac{\ep^2 N p(1-p)}{16 N} \leq \frac{\ep^2}{64}.
\]
\end{proof}

In the next lemma we show a trivial bound for the \emph{mean absolute deviation} of a binomial random variable. 

\begin{lemma}
\label{lemma_exp_vert}
Let $X$ be a random variable following the binomial distribution with parameters $N$ and $p \in [0,1]$.  Then
\[
	\E\bigl[|X - \E[X]|\bigr] \leq \frac{1}{2}\sqrt{N}.
\]
\end{lemma}
\begin{proof}
For any random variable $T$ we have $\Var[T] = \E[T^2] - (\E[T])^2 \geq 0$ thus $(\E[T])^2 \leq \E[T^2]$. Setting $T=|X-\E[X]|$ we get
\[
\left(\E\bigl[|X - \E[X]|\bigr]\right)^2 \leq \E\bigl[(X - \E[X])^2\bigr] = \Var[X],
\]
which gives
\[
\E\bigl[|X - \E[X]|\bigr] \leq \sqrt{\Var[X]} = \sqrt{Np(1-p)} \leq \frac{1}{2} \sqrt{N}.
\]
\end{proof}

\begin{rem}
There are sharper results than Lemma~\ref{lemma_exp_vert}. One can show that for such $X$, we have $\E\bigl[|X - \E[X]|\bigr] = \sqrt{2/\pi} \sqrt{\Var[X]} + O(N^{-1/2})$ (consult \cite{mad_Blyth} or \cite{mad_Berend}). But this improvement of only a constant factor does not yield any improvement to our main theorem.
\end{rem}

The next few lemmas compare the number of components in the two models that we use in this article.

\begin{lemma}
\label{lemma_comps}
Let $t \in \{1,2,\ldots,N\}$. Let $G$ be a graph on $N$ vertices with the maximum degree at most $D$ and $\C_t, C_{t/N}$ be built from the probability space $\Omega$ (as in Section~\ref{sec:formalmodel}). Then for every $\ep>0$
\[
\Pr\left[\C_t > C_{t/N} + \frac{4 D \sqrt{N}}{\ep} \right] \leq \frac{\ep^2}{64}.
\]
\end{lemma}
\begin{proof}
Note that for each $\omega \in \Omega$ we have
\begin{equation}
\label{eq_tilde_nontilde}
\C_t \leq C_{t/N} + D\,|V_{t/N}-\V_t| = C_{t/N} + D\,|V_{t/N} - t|.
\end{equation}
Indeed, the graphs $\tilde{G}_{t}$ and $G_{t/N}$ arose according to the same permutation, induced by $\omega$. Then, they differ by $|V_{t/N}-\V_t|$ vertices and each additional vertex may increase the number of components by at most one and decrease the number of components by at most $D-1$. Recall that $V_{t/N}$ follows the binomial distribution with parameters $N$ and $t/N$ thus $\E[V_{t/N}] = t$ and, by Lemma~\ref{lemma_conc_vert}, we know that
\[
\Pr\left[|V_{t/N} - t| < 4 \sqrt{N}/\ep\right] \geq 1-\frac{\ep^2}{64}.
\]
Therefore,
\[
\Pr\left[\C_t \leq C_{t/N} +4 D \sqrt{N}/\ep\right] \geq 1-\frac{\ep^2}{64}.
\]
\end{proof}

\begin{lemma}
\label{lemma_exp_ineq}
Let $t \in \{1,2,\ldots,N\}$. Let $G$ be a graph on $N$ vertices with the maximum degree bounded by $D$ and $\C_t, C_{t/N}$ be built from the probability space $\Omega$ (as in Section~\ref{sec:formalmodel}). Then
\[
	\E[C_{t/N}] \leq \E[\C_t] + \frac{1}{2} D \sqrt{N}.
\]
\end{lemma}
\begin{proof}
Analogously to \eqref{eq_tilde_nontilde}, for all $\omega \in \Omega$, we write
\[
	C_{t/N} \leq \C_t + D\,|V_{t/N} - t|.
\]
Since $V_{t/N}$ is binomial with expected value $t$, taking expectation (in $\Omega$) on both sides, the conclusion follows from Lemma~\ref{lemma_exp_vert}.
\end{proof}

The next lemma shows that in the blind game we do not loose much if we look at $G_p$ for $p=t/N$ instead of $\G_t$. Also, for a good bound, it is enough to consider a finite number of values of $p$: $1/N, 2/N, ..., N/N$. 

\begin{lemma}
\label{lemma_opt_blind}
Let $G$ be a graph on $N$ vertices with the maximum degree bounded by $D$ and let $\tau^b$ be an optimal algorithm while playing in a blind mode. Then
\begin{equation}\label{eq_maxExp}
	\max_{t \in \{1,\ldots,N\}}\E[C_{t/N}] \leq \E[\C_{\tau^b}] + \frac{1}{2} D \sqrt{N}.
\end{equation}
\end{lemma}

\begin{proof}
In a blind mode the selector does not gain any new information during the game. He can actually decide in advance when to stop. Of course, the only reasonable strategy is to stop at time $t$ maximizing the expected number of components. Therefore,
\[
	\E[\C_{\tau^b}] = \max_{t \in \{1,\ldots,N\}}\E[\C_t].
\]
Now, let $t^b$ be such that $\E[C_{{t^b}/N}] =  \max_{t \in \{1,\ldots,N\}}\E[C_{t/N}]$. It follows that:
\[
	\E[\C_{\tau^b}] = \max_{t \in \{1,\ldots,N\}}\E[\C_t] \geq \E[\C_{t^b}] \ge \E[C_{t^b/N}] - \frac{1}{2} D \sqrt{N},
\]
where the last inequality follows from Lemma~\ref{lemma_exp_ineq}. Now, equation~\eqref{eq_maxExp} follows directly.
\end{proof}

In our final lemma, we show that $C_p$, for constant $p$, is concentrated around its mean for certain~$G$. In particular, the condition in the lemma is satisfied when the maximum degree of $G$ is $o(\sqrt{N})$. 

\begin{lemma}
\label{lemma_comps_con}
Let $p \in [0,1]$. For every $\ep \in (0,1)$ there exists $N_{\ep}$ such that if $G$ is a graph with $N \geq N_{\ep}$ vertices and $\sum_{j=1}^{N} \deg(v_j)^2 \leq \delta_{\ep} N^2$, where $\delta_{\ep} = \frac{\ep^2/64}{\ln{(64/\ep^2)}}$, then $C_p$ satisfies
\[
	\Pr\bigl[C_p \geq \E[C_p] + (\ep/8)N\bigr] \leq \ep^2 /64.
\]
\end{lemma}
\begin{proof}
Let $p\in [0,1]$. For $i \in \{1, 2, \ldots, N\}$ and $\omega \in \Omega$, define $Z_i$ as follows
\begin{align*}
Z_i = \begin{cases} 1 \quad \text{if} \quad \omega_i \leq p,\\
										0 \quad \text{if} \quad \omega_i > p.
\end{cases}
\end{align*}
Recall that $\omega_i$ represents the arrival time of $v_i$. Thus $Z_i$ simply indicates whether $v_i$ belongs to $G_p$ or not. Put $g(Z_1, Z_2, \ldots, Z_N) = C_p$. Note that for two vectors $\mathbf{z}, \mathbf{z'} \in \{0,1\}^N$ that differ only in $j$\textsuperscript{th}  coordinate we have $|g(\mathbf{z})-g(\mathbf{z'})| \leq \deg(v_j)$ unless $v_j$ is isolated in $G$. Indeed, whenever we add a single vertex $v_j$ to the graph, the number of components can increase by at most one or decrease by at most $\deg(v_j)-1$. If $\deg(v_j)=0$ the number of components always increases by one when $v_j$ appears. For $j \in \{1,2,\ldots,N\}$ define
\begin{align*}
b_j = \begin{cases} \deg(v_j) \quad & \text{if} \quad \deg(v_j) > 0,\\
                    1 \quad & \text{if} \quad \deg(v_j) = 0.
\end{cases}
\end{align*}
We have
\[
\sum_{j=1}^{N} b_j^2 \leq N + \sum_{j=1}^{N} \deg(v_j)^2,
\]
where $N$ represents the upper bound for the sum of squared ones while summing over vertices isolated in $G$. By Lemma~\ref{lemma_azuma_ineq}, setting $t=(\ep/8)N$, for sufficiently large $N$ we obtain
\begin{align*}
\Pr\bigl[C_p \geq \E[C_p] + (\ep/8)N\bigr] & \leq \exp\left\{\frac{- 2 (\ep/8)^2 N^2}{N + \sum_{j=1}^{N} \deg(v_j)^2} \right\} \leq \exp\left\{\frac{- 2 \ep^2 N^2}{64 (N + \delta_{\ep}N^2)} \right\} \\
&\leq \exp\left\{\frac{- 2 \ep^2 N^2}{64 \cdot 2 \delta_{\ep}N^2} \right\} = \exp\left\{\frac{-\ep^2}{64 \cdot \delta_{\ep}} \right\} = \ep^2 /64.\qedhere
\end{align*}
\end{proof}

Now we can use the concentration of $C_p$ to prove the main theorem of this section. Hereby we partially follow the lines of the proof of Lemma 3.5 from \cite{kTrees}.

\begin{proof}[Proof of Theorem \ref{thm_blind_full}]
The lower bound is trivial since any stopping algorithm for the blind game is also a stopping algorithm for the full information game.

Now, let $T = \{0, \lfloor\frac{\ep}{8}N\rfloor, \lfloor\frac{2 \ep}{8}N\rfloor, \ldots, \lfloor\frac{\lfloor 8/\ep \rfloor \ep}{8}N\rfloor \} = \{t_0, t_1, \ldots, t_{\lfloor 8/\ep \rfloor}\}$. Put $t_{\lfloor 8/\ep \rfloor + 1} = N+1$. For $i \in \{0,1,\ldots,\lfloor 8/\ep \rfloor\}$ partition $\S$ into $\bigcup S_{t_i}$, where $S_{t_i}$ is the set of permutations $\sigma$ for which $\tau^f(\sigma) \in [t_i,t_{i+1})$. We can write
\[
\E[\C_{\tau^f}] = \frac{1}{N!} \sum_{\sigma \in \S} \C_{\tau^f(\sigma)} = \frac{1}{N!} \sum_{t \in T} \sum_{\sigma \in S_t} \C_{\tau^f(\sigma)},
\]
or, equivalently,
\[
\E[\C_{\tau^f}]  = \int_{\Omega}{\C_{\tau^f(\omega)} \domega} = \sum_{t \in T} \left( \int_{\Omega_{S_t}}{\C_{\tau^f(\omega)} \domega}\right).
\]
Note that for $t \in T$ and $\omega \in \Omega_{S_t}$, letting $\omega \leadsto \sigma$, implies $\sigma \in S_t$ and therefore $\C_{\tau^f(\omega)} = \C_{\tau^f(\sigma)} \leq \C_{t} + \frac{\ep}{8} N + 1$, as each new vertex adds at most one component. By Lemma~\ref{lemma_comps}, for any given $t$, we know that $\Pr\left[\C_t > C_{t/N} + 4 D_{\ep,N}\sqrt{N}/\ep\right] \leq \frac{\ep^2}{64}$ (recall that in the statement of Theorem \ref{thm_blind_full} we define $D_{\ep,N} = \ep^2 \sqrt{N}/32$). As $4 D_{\ep,N}\sqrt{N}/\ep = (\ep/8)N$, this means that 
\[
\C_t \le C_{t/N} + (\ep/8)N
\]
except for at most a $\ep^2/64$ fraction of the whole $\Omega$. One can also easily check that the assumptions of Lemma~\ref{lemma_comps_con} are satisfied; for $\delta_{\ep}$ from Lemma~\ref{lemma_comps_con} we get $\sum_{i=1}^{N}{\deg(v_i)^2} \leq \sum_{i=1}^{N} D_{\ep,N}^2 \leq (\ep^4/32^2) N^2 \leq \delta_{\ep} N^2$. Thus for $N_{\ep}$ from Lemma~\ref{lemma_comps_con}, for any given $t$, we get that for $N \geq N_{\ep}$ at most $\frac{\ep^2}{64}$ fraction of all $\omega$'s do not satisfy the inequality 
\[
C_{t/N} \le \E[C_{t/N}] + (\ep/8) N.
\]
Putting it together we get that for $N \geq N_{\ep}$
\[
\C_t \le \E[C_{t/N}] + (\ep/4) N
\]
except for at most $\frac{\ep^2}{32}$ fraction of the whole $\Omega$, where we can use that $\C_t \le N$. Thus, for each $t \in T$ we get:
\begin{align*}
\int_{\Omega_{S_t}}{\C_{\tau^f(\omega)} \domega} &\le 
\int_{\Omega_{S_t}}{\!\left(\C_{t} + \frac{\ep}{8} N + 1 \right) \domega} \\
&\leq \int_{\Omega_{S_t}}{\!\left( \E[C_{t/N}] + \frac{3\ep}{8} N + 1\right) \domega} + \frac{\ep^2}{32} N.
\end{align*}
Now, summing the above inequality over $t \in T$, for $N \geq N_{\ep}$, and since $\ep<1$, we get
\begin{align*}
\E[\C_{\tau^f}] 
& \leq \sum_{t \in T} \left( \int_{\Omega_{S_t}}{ \left(\E[C_{t/N}] + \frac{3 \ep}{8} N + 1 \right) \domega} \right)  +  \left(\left\lfloor\frac{8}{\ep}\right\rfloor + 1 \right) \frac{\ep^2}{32} N\\
& \leq \sum_{t \in T} \left( \int_{\Omega_{S_t}}{ \left(\E[C_{t/N}] + \frac{3 \ep}{8} N + 1 \right) \domega} \right)  +  \frac{4 \ep}{8}N\\
& \leq \sum_{t \in T} \left( \int_{\Omega_{S_t}}{ \left( \max_{t \in T} \E[C_{t/N}] + \frac{3 \ep}{8} N +1 \right) \domega} \right)  +  \frac{4 \ep}{8}N\\
& = \left( \max_{t \in T} \E[C_{t/N}] + \frac{3 \ep}{8} N + 1\right) \sum_{t \in T} \left( \int_{\Omega_{S_t}}{1\domega}\right)  +  \frac{4 \ep}{8}N\\
& = \max_{t \in T} \E[C_{t/N}] + 1 + \frac{7 \ep}{8} N 
  \leq \max_{t \in \{1,\ldots,N\}} \E[C_{t/N}] + 1 + \frac{7 \ep}{8} N\\
& \leq \E[\C_{\tau^b}] + \frac{1}{2} D_{\ep,N} \sqrt{N} + 1 + \frac{7 \ep}{8} N\\
& \leq \E[\C_{\tau^b}] + \frac{\ep}{8}N + \frac{7 \ep}{8}N 
  = \E[\C_{\tau^b}] + \ep N.
\end{align*}
For the previous to last inequality refer to Lemma~\ref{lemma_opt_blind}. The last inequality follows from the fact that $\frac{1}{2} D_{\ep,N} \sqrt{N} + 1  < (\ep/8) N$ for $\ep<1$.
\end{proof}

\section{Playing on lattices}
In this section we study a particular family of graphs, namely $2$-dimensional lattices: square, triangular and hexagonal one (see Figure \ref{fig_lattices}). Their degrees are bounded by a constant thus all satisfy the assumptions of Theorem \ref{thm_blind_full}. This means that playing a full information game we can not gain significantly more than playing blind. Even though we refer to the three particular lattices here, the reader will notice that the proofs of theorems from this section may be adapted for more general subgraphs of the infinite lattices. The next two lemmas together with Theorem \ref{thm_blind_full} justify that giving the upper and lower bounds just for the value $\E[C_p]$ we get the upper and lower bounds for the gain of either blind or full information game.

\begin{lemma}
\label{lemma_uppBound}
Let $G$ be a graph on $N$ vertices with maximum degree at most $D$. Let $\tau^b$ be the optimal stopping time while playing in a blind mode. Let $g:[0,1] \to \mathbb{R}$ be a function satisfying $\E[C_p] \le N g(p)$ and attaining its maximum at $p_{\max}$. Then
\[
\E[\C_{\tau^b}] \leq N g(p_{\max}) + \frac{1}{2}D\sqrt{N}.
\]
\end{lemma}
\begin{proof}
As in (\ref{eq_tilde_nontilde}), for all $\omega \in \Omega$ we write
\[
\C_{\tau^b} \leq C_{\tau^b/N} + D |V_{\tau^b/N}-\tau^b|.
\]
Since we are considering a blind game, the value of $\tau^b$ is constant. Thus, taking expectation on both sides of the above inequality and using Lemma~\ref{lemma_exp_vert} we get
\begin{align*}
\E[\C_{\tau^b}] & \leq \E[C_{\tau^b/N}] + \frac{1}{2}D \sqrt{N} \leq \max_{t\in\{1,\ldots,N\}} \E[C_{t/N}] + \frac{1}{2}D \sqrt{N} \\
& \leq \sup_{p\in[0,1]} \E[C_p] + \frac{1}{2}D \sqrt{N} \leq N g(p_{\max}) + \frac{1}{2}D \sqrt{N}.
\end{align*}
\end{proof}

\begin{lemma}
\label{lemma_lowBound}
Let $G$ be a graph on $N$ vertices with the maximal degree bounded by $D$. Let $\tau^b$ be the optimal stopping time while playing in a blind mode. Let $f:[0,1]\rightarrow\mathbb{R}$ be a function continuous on $[0,1]$, differentiable on $(0,1)$ such that $N f(p) \leq \E[C_p]$ and attaining its unique maximum at $p_{\max}$. Let also $|f'(p)| < b$ for some constant $b$ and every $p\in(p_{\max}-1/N, p_{\max}+1/N) \cap (0,1)$. Then
\[
\E[\C_{\tau^b}] \geq N f(p_{\max}) -\frac{1}{2}D\sqrt{N}-b.
\]
\end{lemma}
\begin{proof}
By Lemma \ref{lemma_opt_blind}
\[
\E[\C_{\tau^b}]  \geq \max_{t \in \{1,\ldots,N\}}\E[C_{t/N}] - \frac{1}{2} D \sqrt{N} \geq N \cdot \max_{t \in \{1,\ldots,N\}}f(t/N) - \frac{1}{2} D \sqrt{N}.
\]
Let $t^* = \min_{t \in \{1,\ldots,N\}}  \{t: t/N \geq p_{\max}\}$. By the mean value theorem there exists $c \in (p_{\max}-1/N, p_{\max}+1/N) \cap (0,1)$ such that
\[
f(p_{\max}) - f(t^*/N) = |f'(c)|\cdot|t^*/N-p_{\max}| \leq b/N.
\]
Therefore,
\[
\E[\C_{\tau^b}] \geq N f(t^*/N) - \frac{1}{2} D \sqrt{N} \geq N f(p_{\max}) - \frac{1}{2} D \sqrt{N} - b.
\]
\end{proof}

Thus throughout this section we focus on investigating the expected number of components only in the random graph $G_p$. Equivalently, we study parameters of a random graph obtained in a process of site percolation. Site percolation on lattices has been widely and deeply studied. Nevertheless, to the best of our knowledge, the exact values of the expected number of components in a random graph that evolved in a site percolation process are still not known for $2$-dimensional lattices. Some precise upper and lower  bounds were given only for $p$ being close to the critical probability (consult \cite{bollobas2006percolation}). The three theorems of this section give quite tight upper and lower bounds for the maximal value of $\E[C_p]$ in cases when $G$ is a square, a triangular or a hexagonal lattice.

From this point, the graphs $G$ that we consider are always planar. While $C_p$ is a graph parameter that does not depend on a particular embedding of $G$, we will study it indirectly via Euler’s formula, expressing the number of components in terms of the number of vertices, edges and faces of a particular embedding of $G$ in the plane. That is, we study plane graphs.

Whenever we mention a square, triangular or hexagonal lattice we mean their particular embedding as in Figure~\ref{fig_lattices}. For our proofs the exact shape of the outer face is not relevant, but only the fact the number of vertices and faces that are not typical is $o(N)$ (in particular, $o(N)$ of the vertices of $G$ belong to its outer face or have neighbors in the outer face). Other important fact about our embeddings are: in the square lattice (respectively, triangular/hexagonal) all the inner vertices (vertices that do not belong to the outer face) are of degree four (resp., six/three) and all the inner faces are squares (resp. triangles/hexagons).

In this section $F_p$ stands for the number of faces in the plane graph $G_p$. By $F_p^{(k)}$ we denote the number of $k$-faces in the same plane graph, i.e., faces having exactly $k$ edges in the boundary (taking into consideration that if an edge belongs to the boundary of only one face, then it is counted with multiplicity 2 for such face).

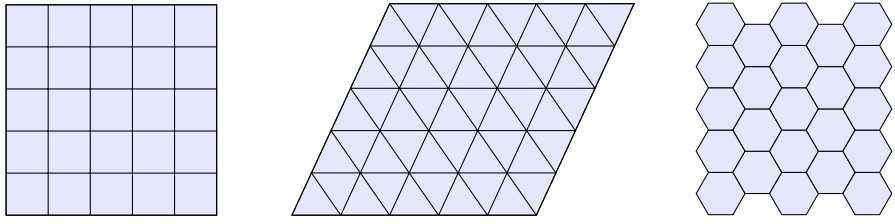
\begin{figure}
\begin{tikzpicture}
\begin{scope}[scale=0.56]
	\draw[fill=blue!10] (0,0) rectangle (5,5);
  \draw (0,0) grid (5,5);
\end{scope}

\begin{scope}[xshift=3.8cm, scale=0.65]
	\newcommand*{\rows}{5}
	\pgfmathsetmacro{\xcoord}{0.4}
	\pgfmathsetmacro{\ycoord}{sin(60)}
    \pgftransformcm{1}{0}{\xcoord}{\ycoord}{\pgfpointorigin} 


    \draw[fill=blue!10] (0,0) rectangle (\rows,\rows);
    \draw (0,0) grid (\rows,\rows);
    \foreach \x in {1,2,...,\rows} {
        \draw (0,\x) -- (\x,0);
        \draw (\rows,\x) -- (\x,\rows);
    } 
\end{scope}

\begin{scope}[xshift=9.5cm, scale=0.65]
\tikzset{hexa/.style= {shape=regular polygon,
                                   regular polygon sides=6,
                                   minimum size=0.65cm, draw,
                                   inner sep=0,anchor=south,
                                   fill=blue!10}}
\foreach \j in {0,...,4}{%
     \ifodd\j 
         \foreach \i in {0,...,3}{\node[hexa] (h\j;\i) at ({\j/2+\j/4},{(\i+1/2)*sin(60)}) {};}        
    \else
         \foreach \i in {0,...,4}{\node[hexa] (h\j;\i) at ({\j/2+\j/4},{\i*sin(60)}) {};}
    \fi}
\end{scope}
\end{tikzpicture}
\caption{Square, triangular and hexagonal lattice.}
\label{fig_lattices}
\end{figure}

\begin{lemma}
\label{lemma_faces}
Let $H$ be a connected, planar graph with $v$ vertices, $e$ edges, with $e \geq 2$, and $f$ faces. Let also $f^{(3)} = f^{(5)} = f^{(7)} =0$, where  $f^{(s)}$ denote the number of $s$-faces in $H$. Then
\[
f \leq (e + 2 f^{(4)} + f^{(6)})/4.
\]
\end{lemma}
\begin{proof}
Since $H$ has at least two edges, we have
\begin{align*}
2 e & = \sum_{k \geq 3 } k f^{(k)} \geq 4 f^{(4)} + 6 f^{(6)} + \sum_{k \geq 8} k f^{(k)} \\
& \geq 4 f^{(4)} + 6 f^{(6)} + 8 (f- f^{(4)} - f^{(6)}) = 8 f - 4 f^{(4)} - 2 f^{(6)}.
\end{align*}
\end{proof}

\begin{thm}
\label{thm_square}
Let $G$ be a square lattice on $n \times n = N$ vertices. Then 
\[
\lim_{N \to \infty} \frac{\sup_{p \in (0,1)} \E[C_p]}{N} \in (0.12953, 0.13268).
\]
\end{thm}
\begin{proof}
Let us start with a lower bound for $\sup_{p \in (0,1)} \E[C_p]$.  The graph $G_p$ is planar for any $p \in (0,1)$. Therefore, by Euler's formula we write
\[
C_p = V_p - E_p + F_p - 1.
\]
We call an inner face \emph{empty} if it does not surround any vertex from $V \setminus V_p$. Note that since $G$ is a square lattice, the number of empty faces in $G_p$ equals $F_p^{(4)}$ or $F_p^{(4)}-1$. The number of inner faces in $G_p$ that surround exactly $k$ vertices from $V \setminus V_p$ will be denoted by $F_p^{[k]}$. For example, all \emph{non-empty} faces (up to rotations) surrounding exactly one or two vertices are shown in Figure~\ref{fig_nonempty_faces}. We have
\[
C_p \geq V_p - E_p + F_p^{(4)} + F_p^{[1]} + F_p^{[2]} - 2.
\]
\begin{figure}
\begin{tikzpicture}
	\draw[step=0.5cm,gray,very thin] (0,0) grid (1,1);
	\foreach \x in {0,0.5,1}{
		\foreach \y in {0,1}{
			\node[draw,circle,inner sep=1.5pt, fill] at (\x,\y){};
		}
	}
	\node[draw,circle,inner sep=1.5pt, fill] at (0,0.5){};
	\node[draw,circle,inner sep=1.5pt, fill] at (1,0.5){};
	\node[draw,circle,inner sep=1.5pt, fill=white] at (0.5,0.5){};
	\draw (0,0) rectangle (1,1);
	
	\draw[step=0.5cm,gray,very thin] (3,0) grid (4.5,1);
	\foreach \x in {3,3.5,4,4,4.5}{
		\foreach \y in {0,1}{
			\node[draw,circle,inner sep=1.5pt, fill] at (\x,\y){};
		}
	}
	\node[draw,circle,inner sep=1.5pt, fill] at (3,0.5){};
	\node[draw,circle,inner sep=1.5pt, fill=white] at (3.5,0.5){};
	\node[draw,circle,inner sep=1.5pt, fill=white] at (4,0.5){};
  \node[draw,circle,inner sep=1.5pt, fill] at (4.5,0.5){};
	\draw (3,0) rectangle (4.5,1);
	
	\draw[step=0.5cm,gray,very thin] (6.5,0) grid (8,1.5);
	\draw[gray,very thin] (6.5,1) -- (6.5,1.5);
	\draw[thick] (7,1.5) -- (8,1.5);
	\draw[thick] (6.5,1) -- (7,1);
	\draw[thick] (7.5,0.5) -- (8,0.5);
	\draw[thick] (6.5,0) -- (7.5,0);
	\draw[thick] (6.5,0) -- (6.5,1);
	\draw[thick] (7,1) -- (7,1.5);
	\draw[thick] (7.5,0) -- (7.5,0.5);
	\draw[thick] (8,0.5) -- (8,1.5);
	\node[draw,circle,inner sep=1.5pt, fill=white] at (7,0.5){};
	\node[draw,circle,inner sep=1.5pt, fill=white] at (7.5,1){};
	
	\foreach \x in {7,7.5,8}{
		\node[draw,circle,inner sep=1.5pt, fill] at (\x,1.5){};
	}
	\foreach \x in {6.5,7,8}{
		\node[draw,circle,inner sep=1.5pt, fill] at (\x,1){};
	}
	\foreach \x in {6.5,7.5,8}{
		\node[draw,circle,inner sep=1.5pt, fill] at (\x,0.5){};
	}
	\foreach \x in {6.5,7,7.5}{
		\node[draw,circle,inner sep=1.5pt, fill] at (\x,0){};
	}
\end{tikzpicture}
\caption{Examples of faces in $G_p$ surrounding $1$ or $2$ (white) vertices  from $V \setminus V_p$.}
\label{fig_nonempty_faces}
\end{figure}
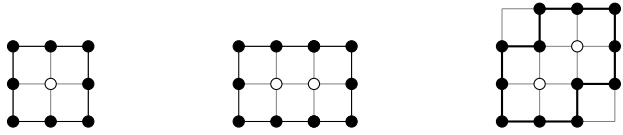


\noindent
Now,
\begin{align*}
& \E[V_p] = N p, \quad \E[E_p] = (2 N - 2 n) p^2 \leq 2Np^2,\\ & \E[F_p^{(4)}] = (n-1)^2 p^4 + o(N) = N p^4 + o(N),\\
& \E[F_p^{[1]}] = Np^8(1-p)+o(N) \quad \textnormal{and}\\
& \E[F_p^{[2]}] = 2Np^{10}(1-p)^2 + 2Np^{12}(1-p)^2 + o(N).
\end{align*}
Therefore,
\begin{align*}
\E[C_p] & \geq N(p - 2p^2 + p^4 + p^8(1-p) + 2p^{10}(1-p)^2 + 2p^{12}(1-p)^2 ) + o(N) \\
& = N(p-2p^2+p^4+p^8-p^9+2p^{10}-4p^{11}+2p^{12}-4p^{13}+2p^{14}) + o(N) \\
& = N f(p) + o(N).
\end{align*}
The function $f(p)$ attains its unique maximum on $[0,1]$ at $p_{\max}$ approximately $0.27$ and \mbox{$f(p_{\max}) > 0.12953$}.

Now we calculate an upper bound for $\sup_{p \in (0,1)} \E[C_p]$. Graph $G_p$ has $C_p$ components, name them $H_1, H_2, \ldots, H_{C_p}$. For $i \in \{1, 2, \ldots, C_p\}$ let $e_i$ denote the number of edges in $H_i$ and $f_i$ the number of faces in $H_i$ if $H_i$ is considered as a standalone graph. Let $J_p$ be the set of indices $j \in \{1,2,\dots,C_p\}$ such that $H_j$ has at least two edges. Note that $|J_p| = C_p - \iV_p - \iE_p$ where $\iV_p$ stands for the number of isolated vertices in $G_p$ and $\iE_p$ for the number of isolated edges. 
By Lemma \ref{lemma_faces}, we know that $f_j \leq (e_j + 2 f_j^{(4)} + f_j^{(6)})/4$ for each $j \in J_p$. Therefore
\[
\sum_{j \in J_p} f_j \leq \frac{1}{4}\sum_{j \in J_p} e_j + \frac{1}{2}\sum_{j \in J_p} f_j^{(4)} + \frac{1}{4}\sum_{j \in J_p} f_j^{(6)}.
\]
We have $\sum_{j \in J_p} e_j = E_p - \iE_p$ and $\sum_{j \in J_p} f_j = F_p + |J_p|-1$; indeed, the outer face of $G_p$ was counted once in every $f_j$. Denote $F_p^{(4*)} = \sum_{j \in J_p} f_j^{(4)}$ and $F_p^{(6*)} = \sum_{j \in J_p} f_j^{(6)}$. We have
\[
F_p + |J_p|-1 \leq \frac{1}{4}(E_p - \iE_p) + \frac{1}{2}F_p^{(4*)} + \frac{1}{4}F_p^{(6*)}
\]
and since $|J_p| = C_p - \iV_p - \iE_p$ we obtain
\[
F_p \leq \frac{1}{4} E_p + \frac{3}{4} \iE_p + \iV_p + \frac{1}{2} F_p^{(4*)} + \frac{1}{4} F_p^{(6*)} - C_p + 1.
\]
Using it together with the Euler's formula, $C_p = V_p-E_p+F_p-1$, we get
\begin{equation} \label{eq:Cp_upp}
C_p \leq \frac{1}{2}\left(V_p - \frac{3}{4}E_p + \frac{3}{4}\iE_p + \iV_p + \frac{1}{2}F_p^{(4*)} + \frac{1}{4}F_p^{(6*)}\right).
\end{equation}
The number of isolated vertices and the number of isolated edges in $G_p$ satisfy $\E[\iV_p] = N p(1-p)^4 + o(N)$ and $\E[\iE_p] = 2Np^2(1-p)^6+o(N)$. Consider $F_p^{(4*)} = \sum_{j \in J_p} f_j^{(4)}$. Note that, unless $|J_p| =1$ (in which case $F_p^{(4*)}=F_p^{(4)}$ or $F_p^{(4*)}=F_p^{(4)}+1$), the value of $F_p^{(4*)}$ is equal to the number of $4$-faces in $G_p$ plus the number of graphs $H_j$ whose outer face is a $4$-face. In Figure \ref{fig_4_faces} we consider all cases, up to rotation, where the outer face of a component is a $4$-face. Therefore, recalling that $\E[F^{(4)}_p] = Np^4 + o(N)$,
\[
\E[F_p^{(4*)}] = Np^4 +Np^4(1-p)^8+4Np^3(1-p)^7 + 2Np^3(1-p)^8 + o(N).
\]
Now, consider $F_p^{(6*)} = \sum_{j \in J_p} f_j^{(6)}$.  Note that $f_j^{(6)} \neq 0$ if, and only if, $f_j^{(6)} = 1$ and the outer face of $H_j$ is a $6$-face (see examples in Figure \ref{fig_6_faces}). We get
\[
\E[F_p^{(6*)}] = 8Np^4(1-p)^8 + 8Np^4(1-p)^9+2Np^4(1-p)^{10}+8Np^5(1-p)^9 + o(N).
\]

\begin{figure}
\begin{tikzpicture}
	\draw[step=0.5cm,gray,very thin] (0,0) grid (1.5,1.5);
	\foreach \x in {0.5,1}{
		\foreach \y in {0,1.5}{
			\node[draw,circle,inner sep=1.5pt, fill=white] at (\x,\y){};
		}
	}
	\foreach \x in {0,1.5}{
		\foreach \y in {0.5,1}{
			\node[draw,circle,inner sep=1.5pt, fill=white] at (\x,\y){};
		}
	}
	\foreach \x in {0.5,1}{
		\foreach \y in {0.5,1}{
			\node[draw,circle,inner sep=1.5pt, fill] at (\x,\y){};
		}
	}
	\draw (0.5,0.5) rectangle (1,1);
	
	\draw[step=0.5cm,gray,very thin] (3.5,0) grid (5,1.5);
	\draw[gray, very thin] (3.5,0)--(3.5,1.5);
	
	\foreach \x in {3.5,4,4.5}{
		\node[draw,circle,inner sep=1.5pt, fill=white] at (\x,\x-3){};
		\node[draw,circle,inner sep=1.5pt, fill=white] at (\x+0.5,\x-3.5){};
	}
	\foreach \x in {4.5,5}{
		\node[draw,circle,inner sep=1.5pt, fill=white] at (\x,\x-4.5){};
	}
	\node[draw,circle,inner sep=1.5pt, fill] at (4,0.5){};
	\node[draw,circle,inner sep=1.5pt, fill] at (4.5,0.5){};
	\node[draw,circle,inner sep=1.5pt, fill] at (4.5,1){};
	
	\draw[thick] (4,0.5)--(4.5,0.5);
	\draw[thick] (4.5,0.5)--(4.5,1);
	
	\draw[step=0.5cm,gray,very thin] (7,0) grid (9,1);
	\draw[gray, very thin] (7,0)--(7,1);
	
	\foreach \x in {7.5,8,8.5}{
		\foreach \y in {0,1}{
			\node[draw,circle,inner sep=1.5pt, fill=white] at (\x,\y){};
		}
		\node[draw,circle,inner sep=1.5pt,fill] at (\x,0.5){};
	}
	\foreach \x in {7,9}
		\node[draw,circle,inner sep=1.5pt, fill=white] at (\x,0.5){};
	
	\draw[thick] (7.5,0.5)--(8.5,0.5);
	
\end{tikzpicture}
\caption{Examples of $H_j$'s having the outer face being a $4$-face. Black vertices belong to the face, white ones must be on $V\setminus V_p$ and non-marked vertices are irrelevant.}
\label{fig_4_faces}
\end{figure}
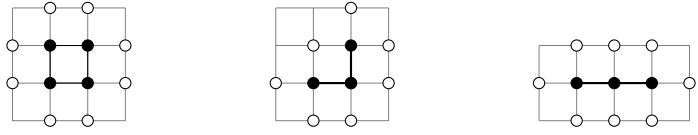

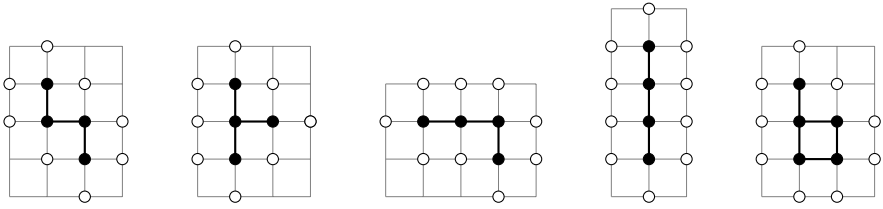
\begin{figure}
\begin{tikzpicture}
	\draw[step=0.5cm,gray,very thin] (8,0) grid (9,2.5);
	\draw[gray, very thin] (8,0)--(8,2.5);	
	\foreach \x in {8,9}
		\foreach \y in {0.5,1,1.5,2}
			\node[draw,circle,inner sep=1.5pt, fill=white] at (\x,\y){};
	\foreach \y in {0.5,1,1.5,2}
			\node[draw,circle,inner sep=1.5pt, fill] at (8.5,\y){};
	\node[draw,circle,inner sep=1.5pt, fill=white] at (8.5,0){};
	\node[draw,circle,inner sep=1.5pt, fill=white] at (8.5,2.5){};
	\draw[thick] (8.5,0.5)--(8.5,2);
	
	\begin{scope}[xshift=-3.5cm]
	\draw[step=0.5cm,gray,very thin] (3.5,0) grid (5,2);
	\draw[gray, very thin] (3.5,0)--(3.5,2);
	\foreach \x in {4,4.5,5} {
		\node[draw,circle,inner sep=1.5pt, fill=white] at (\x,-\x+6){};
		\node[draw,circle,inner sep=1.5pt, fill=white] at (\x-0.5,-\x+5){};
	}
	\node[draw,circle,inner sep=1.5pt, fill=white] at (3.5,1.5){};
	\node[draw,circle,inner sep=1.5pt, fill=white] at (5,0.5){};
	\foreach \x in {4,4.5} {
		\node[draw,circle,inner sep=1.5pt,fill] at (\x,-\x+5.5){};
		\node[draw,circle,inner sep=1.5pt,fill] at (\x,-\x+5){};
	}
	\draw[thick](4,1)--(4,1.5);
	\draw[thick](4,1)--(4.5,1);
	\draw[thick](4.5,0.5)--(4.5,1);
	
	\draw[step=0.5cm,gray,very thin] (6,0) grid (7.5,2);
	\draw[gray, very thin] (6,0)--(6,2);
	\foreach \y in {0.5,1,1.5}
		\node[draw,circle,inner sep=1.5pt, fill=white] at (6,\y){};
	\foreach \x in {6.5,7,7.5} {
		\node[draw,circle,inner sep=1.5pt, fill=white] at (\x,\x-6.5){};
		\node[draw,circle,inner sep=1.5pt, fill=white] at (\x,-\x+8.5){};
	}
	\foreach \y in {0.5,1,1.5}
		\node[draw,circle,inner sep=1.5pt,fill] at (6.5,\y){};
	\node[draw,circle,inner sep=1.5pt,fill] at (7,1){};
	\draw[thick] (6.5,0.5)--(6.5,1.5);
	\draw[thick] (6.5,1)--(7,1);	

	\draw[step=0.5cm,gray,very thin] (8.5,0) grid (10.5,1.5);
	\draw[gray, very thin] (8.5,0)--(8.5,1.5);
	\foreach \x in {9,9.5,10}
		\node[draw,circle,inner sep=1.5pt, fill=white] at (\x,1.5){};
	\foreach \x in {8.5,10.5}
		\node[draw,circle,inner sep=1.5pt, fill=white] at (\x,1){};
	\foreach \x in {9,9.5,10.5}
		\node[draw,circle,inner sep=1.5pt, fill=white] at (\x,0.5){};
	\node[draw,circle,inner sep=1.5pt, fill=white] at (10,0){};
	\foreach \x in {9,9.5,10}
		\node[draw,circle,inner sep=1.5pt,fill] at (\x,1){};
	\node[draw,circle,inner sep=1.5pt,fill] at (10,0.5){};
	\draw[thick] (9,1)--(10,1);
	\draw[thick] (10,1)--(10,0.5);
	\end{scope}
	
	\draw[step=0.5cm,gray,very thin] (10,0) grid (11.5,2);
	\draw[gray, very thin] (10,0)--(10,2);
	\foreach \y in {0.5,1,1.5}
		\node[draw,circle,inner sep=1.5pt, fill=white] at (10,\y){};
	\foreach \y in {0.5,1,1.5}
		\node[draw,circle,inner sep=1.5pt, fill] at (10.5,\y){};
	\foreach \y in {0.5,1}
		\node[draw,circle,inner sep=1.5pt, fill] at (11,\y){};
	\foreach \x in {10.5,11}
		\node[draw,circle,inner sep=1.5pt, fill=white] at (\x,0){};					
	\foreach \y in {0.5,1}
		\node[draw,circle,inner sep=1.5pt, fill=white] at (11.5,\y){};	
	\node[draw,circle,inner sep=1.5pt, fill=white] at (10.5,2){};
	\node[draw,circle,inner sep=1.5pt, fill=white] at (11,1.5){};
	\draw[thick] (10.5,0.5)--(10.5,1.5);
	\draw[thick] (11,0.5)--(11,1);
	\draw[thick] (10.5,0.5)--(11,0.5);
	\draw[thick] (10.5,1)--(11,1);
\end{tikzpicture}
\caption{Examples of $H_j$'s having the outer face being a $6$-face. Black vertices belong to the face, white ones must be on $V\setminus V_p$ and non-marked vertices are irrelevant.}
\label{fig_6_faces}
\end{figure}

Finally,
\begin{align*}
\E[C_p] & \leq \frac{1}{2}N\left(p - \frac{3}{2}p^2 + \frac{3}{2}p^2(1-p)^6 + p(1-p)^4 +\frac{1}{2}p^4 + \frac{1}{2}p^4(1-p)^8 + 2p^3(1-p)^7 \right.\\
&  \quad \left. \mathrel{+} p^3(1-p)^8+ 2p^4(1-p)^8 + 2p^4(1-p)^9+\frac{1}{2}p^4(1-p)^{10}+2p^5(1-p)^9\right) + o(N) \\
& = N\left(p-2 p^2+p^4 + \frac{43}{2}p^6 - \frac{165}{2}p^7 + \frac{535}{4}p^8 - 112p^9 + \frac{81}{2}p^{10} \right. \\
&  \quad \left. \mathrel{+} \frac{17}{2}p^{11} - \frac{29}{2}p^{12}+\frac{11}{2}p^{13}-\frac{3}{4}p^{14}\right) + o(N) = N g(p)+o(N).
\end{align*}
The function $g(p)$ attains its maximum on $[0,1]$ at $p'_{\max}$ approximately $0.29$ and \mbox{$g(p'_{\max}) < 0.13268$}.
\end{proof}

\begin{rem}
Graph $G$, being a square lattice on $n \times n = N$ vertices, is $2$-degenerate. It can be easily transformed into a maximal $2$-degenerate graph $G'$ by adding $2n-3 = \Theta(\sqrt{N})$ edges (note that a $2$-degenerate graph on $n^2$ vertices has at most $2n^2-3$ edges while the lattice has exactly $2n^2-2n$ edges). Note that Theorem \ref{thm_square} still holds for $G'$ since these additional edges may decrease the number of gained components by at most $o(N)$. In \cite{kTrees} it was proved that $\left(\frac{k^k}{(k+1)^{k+1}}+o(1)\right)N$ constitutes the upper bound for the expected number of components while playing full information game on a maximal $k$-degenerate graph on $N$ vertices. Examples of graphs that attain this maximum are $k$-trees. Note that $G'$ serves as an example of a maximal $2$-degenerate graph which does not attain this upper bound. Indeed $ 0.13268 < 2^2/3^3 \approx 0.148$. One difference between $G'$ and a $k$-tree is its unbounded treewidth. However, we do not know how (and whether) this parameter really influences the gain of the game.
\end{rem}

\begin{thm}
\label{thm_triang}
Let $G$ be a triangular lattice on $N$ vertices (drawn as in Figure~\ref{fig_lattices}). Then 
\[
\lim_{N \to \infty} \frac{\sup_{p \in (0,1)} \E[C_p]}{N} \in (0.09629, 0.10107).
\]
\end{thm}
\begin{proof}
For the lower bound for $\sup_{p \in (0,1)} \E[C_p]$ we write
\[
C_p \geq V_p - E_p + F_p^{(3)} + F_p^{(6)} - 1.
\]
(Note that $F_p^{(4)}$ and $F_p^{(5)}$ are either $0$ or $1$, being $1$ only if they count the outer face of $G_p$, so we do not loose much by ignoring those terms). For $G$ being a triangular lattice we have $\E[V_p] = Np$, $\E[E_p] = 3Np^2 + o(N)$, $\E[F_p^{(3)}] = 2Np^3 + o(N)$ and $\E[F_p^{(6)}] = Np^6(1-p) + o(N)$ (where the $(1-p)$ in the last expression comes from the fact that the vertex of $G$ that belongs to the interior of the hexagon cannot be selection to $G_p$). Thus
\[
\E[C_p] \geq N(p-3p^2+2p^3+p^6(1-p)) + o(N) = Nf(p) + o(N).
\]
The function $f(p)$ attains its unique maximum on $[0,1]$ at $p_{\max}$ approximately $0.21$ and \mbox{$f(p_{\max}) > 0.09629$}.

For the upper bound we follow the lines of the proof of Theorem \ref{thm_square} just this time for the number of edges in $H_j$ we use the following inequality
\[
2 e_j \geq 3f_j^{(3)} + 4f_j^{(4)} + 5f_j^{(5)} + 6(f_j - f_j^{(3)} - f_j^{(4)} - f_j^{(5)}).
\]
Hence
\[
\sum_{j \in J_p} f_j \leq \frac{1}{3}\sum_{j \in J_p} e_j + \frac{1}{2}\sum_{j \in J_p} f_j^{(3)} + \frac{1}{3}\sum_{j \in J_p} f_j^{(4)} + \frac{1}{6}\sum_{j \in J_p} f_j^{(5)}.
\]
Again denote $F_p^{(3*)} = \sum_{j \in J_p} f_j^{(3)}$,  $F_p^{(4*)} = \sum_{j \in J_p} f_j^{(4)}$ and $F_p^{(5*)} = \sum_{j \in J_p} f_j^{(5)}$. By Euler's formula applied to $G_p$, we have
\[
C_p \leq \frac{1}{2}\left(V_p - \frac{2}{3}E_p + \frac{2}{3}\iE_p + \iV_p + \frac{1}{2}F_p^{(3*)} + \frac{1}{3}F_p^{(4*)} + \frac{1}{6}F_p^{(5*)}\right).
\]
The fact that
\begin{align*}
& \E[\iE_p] = 3Np^2(1-p)^8 + o(N), \quad\quad \E[\iV_p] = Np(1-p)^6 + o(N),\\
& \E[F_p^{(3*)}] = 2Np^3+2Np^3(1-p)^9 + o(N),\\
& \E[F_p^{(4*)}] = 9Np^3(1-p)^{10} + 3Np^4(1-p)^{10} + o(N) \quad \textnormal{and}\\
& \E[F_p^{(5*)}] = 12 Np^4(1-p)^{11} + 6Np^5(1-p)^{11} + o(N)
\end{align*}
implies
\begin{align*}
\E[C_p] & \leq \frac{1}{2}N\bigg(p-2p^2 +2p^2(1-p)^8 + p(1-p)^6 + p^3 + p^3(1-p)^9  \bigg. \\
& \left. + 3p^3(1-p)^{10} + p^4(1-p)^{10} + 2p^4(1-p)^{11} + p^5(1-p)^{11}\right) + o(N) \\
& = N g(p) + o(N).
\end{align*}
The function $g(p)$ attains its maximum on $[0,1]$ at $p'_{\max}$ approximately $0.24$ and \mbox{$g(p'_{\max}) < 0.10107$}.
\end{proof}

\begin{rem}
Graph $G$, which is a triangular lattice on $N$ vertices, is $3$-degenerate. Again, it can be easily transformed into a maximal $3$-degenerate graph $G'$ by adding $\Theta(\sqrt{N})$ edges. Such $G'$ serves as another example of a maximal $k$-degenerate graph which does not attain the upper bound from \cite{kTrees}. Indeed $0.10107 < 3^3/4^4 \approx 0.1055$.
\end{rem}

\begin{thm}
\label{thm_hex}
Let $G$ be a hexagonal lattice on $N$ vertices (drawn as in Figure~\ref{fig_lattices}). Then 
\[
\lim_{N \to \infty} \frac{\sup_{p \in (0,1)} \E[C_p]}{N} \in (0.16738, 0.17144).
\]
\end{thm}
\begin{proof}
For $G$ being a hexagonal lattice we have $\E[V_p] = Np$, $\E[E_p] = (3/2)N p^2 + o(N)$, $\E[F_p^{(6)}] = (N/2)p^6  + o(N)$ and $\E[F_p^{(12)}] = N p^{12}(1-p)  + o(N)$. This time for the lower bound we write
\[
C_p \geq V_p - E_p + F_p^{(6)} + F_p^{(12)} - 1
\]
thus
\[
\E[C_p] \geq N\left(p-\frac{3}{2}p^2 + \frac{1}{2}p^6 + p^{12}(1-p)\right) +o(N) = N f(p) + o(N).
\]
The function $f(p)$ attains its unique maximum on $[0,1]$ at $p_{\max}$ approximately $0.34$ and \mbox{$f(p_{\max}) > 0.16738$}.

For the upper bound we use again the inequality (\ref{eq:Cp_upp}). Since now $\E[F_p^{(4*)}] = 3Np^3(1-p)^5 + o(N)$ and $\E[F_p^{(6*)}] = (N/2)p^6 + (N/2)p^6(1-p)^6 + Np^4(1-p)^6 + 6Np^4(1-p)^6 + o(N)$
we get
\begin{align*}
\E[C_p] & \leq \frac{1}{2}N\left(p-\frac{9}{8}p^2 + \frac{9}{8}p^2(1-p)^4 + p(1-p)^3 + \frac{3}{2}p^3(1-p)^5 \right.\\
& \quad \left. + \frac{1}{8}p^6 + \frac{1}{8}p^6(1-p)^6+\frac{1}{4}p^4(1-p)^6+\frac{3}{2}p^4(1-p)^6 \right) + o(N) \\
& = N g(p) + o(N).
\end{align*}
The function $g(p)$ attains its maximum on $[0,1]$ at $p'_{\max}$ approximately $0.36$ and \mbox{$g(p'_{\max}) < 0.17144$}.
\end{proof}

The following corollary summarizes the results for lattices.

\begin{cor}\label{cor:lattices}
Let $\tau^b$ be the optimal stopping time while playing a blind game and $\tau^f$ be the optimal stopping time while playing a full information game on a lattice with $N$ vertices. Let $c^b = (1/N)\E[\C_{\tau^b}]$ and $c^f = (1/N)\E[\C_{\tau^f}]$. By Theorems~\ref{thm_square}, \ref{thm_triang} and~\ref{thm_hex} together with Lemmas~\ref{lemma_uppBound} and~\ref{lemma_lowBound}, for sufficiently large $N$, we have:
\begin{center}
\begin{tabular}{ccccc}
\toprule
\textbf{Lattice}& \textbf{Lower bound for} $c^b$  &\textbf{Upper bound for} $c^b$ & \textbf{gap}\\
\midrule
{square}     & $0.12953$ & $0.13268$ & $0.00315$\\
{triangular} & $0.09629$ & $0.10107$ & $0.00478$\\
{hexagonal}  & $0.16738$ & $0.17144$ & $0.00406$\\
\bottomrule
\end{tabular}
\end{center}
Furthermore, for every $\ep \in (0,1)$ and for sufficiently large $N$, by Theorem~\ref{thm_blind_full}  we have:
\[
	c^b \le c^f \le c^b + \ep.
\]
\end{cor}

\section{Final comments and questions}

Corollary~\ref{cor:lattices} presents tight bounds for the expected number of components that an optimal blind strategy gains for each lattice and shows that with full information the gain is almost the same. It does not find the exact optimal stopping time, but one can easily verify that if the player (blindly) stops at time $\tau = \lfloor p_{\max} N \rfloor$, where $p_{\max}$ is a value of $p$ that maximizes one of the functions $f(p)$ used in the proofs for the lower bound in Theorems~\ref{thm_square}, \ref{thm_triang} or~\ref{thm_hex}, then, in the respective lattice, the value of $(1/N)\E[\C_{\tau}]$ belongs to the interval given in Corollary~\ref{cor:lattices}.

Interpreting the proofs of Theorems~\ref{thm_square}, \ref{thm_triang} and~\ref{thm_hex}, one concludes that at the moment when the expected number of components is maximized, most faces are very small. However, if we take into account only the expected number of isolated vertices and isolates edges while counting components, we would get much  worse lower bound for the expected number of components (say, of order $0.103N$ in the square lattice). This indicates that the number of small faces is indeed relevant. On the other hand, we tried to include slightly larger faces in our proof, what yielded a longer case analysis, but ended up with no significant improvements. 

A natural open question is whether one can relax the condition about the maximum degree in Theorem~\ref{thm_blind_full}, e.g., to graphs with the maximum degree of order $o(N)$. It is also of interest to determine for which graphs the optimal blind strategy returns the linear expected number of components. This would precisely define the family of graphs for which the bounds from Theorem~\ref{thm_blind_full} are useful.
\\

\paragraph{\textbf{Acknowledgements}}
This research was partially supported by Polish National Science Center - grant MINIATURA 3, DEC-2019/03/X/ST6/00657, by the Coordenação de Aperfeiçoamento de Pessoal de Nível Superior - Brasil (CAPES) - 88881.569474/2020-01, by CNPQ (Proc.437841/2018-9, Proc.314374/2018-3).


\bibliographystyle{plain}
\bibliography{maxcomps}

\end{document}